\def\timestamp{%
Time-stamp: <copies-of-N.tex: dinsdag 31-10-2023 at 17:13:31 (cet)>}
\def\stripname Time-stamp: <#1: #2 #3 at #4 #5>{#3/#4 (#1)}
\edef\filedate{\expandafter\stripname\timestamp}
\documentclass[a4paper]
              {amsart}

\DeclareMathSymbol\A \mathord{AMSb}{`A}
\DeclareMathSymbol\BB \mathord{AMSb}{`B}
\DeclareMathSymbol\D \mathord{AMSb}{`D}
\DeclareMathSymbol\E \mathord{AMSb}{`E}
\DeclareMathSymbol\N \mathord{AMSb}{`N}
\DeclareMathSymbol\R \mathord{AMSb}{`R}
\newcommand\calA{\mathcal{A}}
\newcommand\calF{\mathcal{F}}
\newcommand\calI{\mathcal{I}}

\newcommand\calR{\mathcal{R}}
\newcommand\calS{\mathcal{S}}
\newcommand\calT{\mathcal{T}}
\newcommand\calZ{\mathcal{Z}}
\newcommand\cee{\mathfrak{c}}
\newcommand\dee{\mathfrak{d}}
\newcommand\pee{\mathfrak{p}}
\newcommand\arr{\mathfrak{r}}
\newcommand\uff{\mathfrak{u}}
\let\axiom\mathsf
\newcommand\CH{\axiom{CH}}
\newcommand\notCH{\neg\axiom{CH}}
\newcommand\MA{\axiom{MA}}

\newcommand\MAnegCH{\MA+\neg\CH}
\newcommand\cl{\operatorname{cl}}
\newcommand\Int{\operatorname{int}}

\newcommand\orpr[2]{\langle{#1},{#2}\rangle}
\newcommand\preim{^\gets}
\newcommand\omegaseq[2][n]{\langle{#2}_{#1}:{#1}\in\omega\rangle}
\newcommand\omegaoneseq[2][\alpha]{\langle{#2}_{#1}:{#1}\in\omega_1\rangle}
\DeclareMathSymbol\restr\mathbin{AMSa}{"16}
\DeclareMathSymbol\le    \mathrel{AMSa}{"36}
\DeclareMathSymbol\ge    \mathrel{AMSa}{"3E}
\mathchardef\mathhyphen "002D   
\newcommand\ulim[1][u]{#1\mathhyphen\!\lim}
\newcommand\Rwone{\R^{\omega_1}}
\newcommand\finseq{{\vphantom\omega}^{<\omega_1}\omega}
\theoremstyle{plain}
\newtheorem{theorem}{Theorem}[section]

\newtheorem{proposition}[theorem]{Proposition}
\theoremstyle{definition}
\newtheorem*{assumption}{Assumption}
\theoremstyle{remark}
\newtheorem{question}{Question}
\newtheorem{remark}[theorem]{Remark}
\usepackage{amsrefs}

\begin{document}

\title{Closed copies of $\N$ in $\Rwone$}

\author[A. Dow]{Alan Dow}
\address{Department of Mathematics\\
         UNC-Charlotte\\
         9201 University City Blvd. \\
         Charlotte, NC 28223-0001}
\email{adow@charlotte.edu}
\urladdr{https://webpages.uncc.edu/adow}

\author[K. P. Hart]{Klaas Pieter Hart}
\address{Faculty EEMCS\\TU Delft\\
         Postbus 5031\\2600~GA {} Delft\\the Netherlands}
\email{k.p.hart@tudelft.nl}
\urladdr{https://fa.ewi.tudelft.nl/\~{}hart}

\author[J. van Mill]{Jan van Mill}
\address{KdV Institute for Mathematics\\
         University of Amsterdam\\
         P.O. Box 94248\\
         1090~GE {} Amsterdam\\
         The Netherlands}
\email{j.vanmill@uva.nl}
\urladdr{https://staff.fnwi.uva.nl/j.vanmill/}

\author[J. Vermeer]{Hans Vermeer}
\address{Faculty EEMCS\\TU Delft\\
         Postbus 5031\\2600~GA {} Delft\\the Netherlands}
\email{j.vermeer@tudelft.nl}

\begin{abstract}
We investigate closed copies of~$\N$ in powers of~$\R$ with
respect to $C^*$- and $C$-embedding.  
We show that $\R^{\omega_1}$ contains closed copies of~$\N$ that are not
$C^*$-embedded.
\end{abstract}


\dedicatory{To Istv\'an Juh\'asz on his 80th birthday}

\date\filedate
\maketitle

\section*{Introduction}

In~\cite{dhvmv2022} we presented examples of realcompact spaces that
have closed subsets that are $C^*$-embedded but not $C$-embedded.

One of these spaces, call it $X$, even contains a closed copy of~$\N$ 
(the discrete space of natural numbers) 
that is $C^*$-embedded but not $C$-embedded.
It is well known that the diagonal map from~$X$ into~$\R^{C(X)}$ embeds $X$ 
as a closed $C$-embedded subspace.
The closed copy of~$\N$ in~$X$ then becomes a closed copy of~$\N$ in~$\R^{C(X)}$
that is $C^*$-embedded but not $C$-embedded. 

An intermediate realcompact space, $Y$ say, contains a closed copy of~$\N$ 
that is not $C^*$-embedded and, as above, this yields a closed copy of~$\N$
in~$\R^{C(Y)}$ that is not $C^*$-embedded.

For both spaces the cardinality of the set of continuous functions is equal 
to~$\cee$, which yields the interesting result that one can find closed
copies of~$\N$ in~$\R^\cee$, that are not $C^*$-embedded, and that are
$C^*$-embedded but not $C$-embedded.

In the first version of~\cite{dhvmv2022} we posed two questions suggested 
by these results.
We repeat them here.

\begin{question}\label{question.1}
What is the minimum cardinal~$\kappa$ such that $\R^\kappa$ contains
a closed copy of~$\N$ that is $C^*$-embedded but not $C$-embedded?  
\end{question}

\begin{question}\label{question.2}
What is the minimum cardinal~$\kappa$ such that $\R^\kappa$ contains
a closed copy of~$\N$ that is not $C^*$-embedded?  
\end{question}

Given that $\R^{\omega_0}$ is metrizable and we know that in both cases
we have $\aleph_0<\kappa\le\cee$.

After we posted the first version of the present paper on \texttt{arxiv.org}
Roman Pol kindly drew our attention to three papers,
\cite{MR3003330}, \cite{MR3239211}, and~\cite{MR4138431},
containing results that address the two questions above.

These are:
\begin{enumerate}
\item The main result, Theorem~10, of~\cite{MR3003330} implies that there
      are many closed copies of~$\N$ in~$\R^{\omega_1}$ that are not 
      $C^*$-embedded.
\item The paper~\cite{MR3239211} contains another example, Example~1.1,
      of a closed copy of~$\N$ in~$\R^{\omega_1}$ that is not $C^*$-embedded,
      and an example of a closed copy of~$\N$ in~$\R^\cee$ that
      is $C^*$-embedded but not $C$-embedded.
\item In~\cite{MR4138431} one finds a result, Theorem~3.1, that implies that
      under the assumption of the inequality $\arr>\aleph_1$ every
      $C^*$-embedded subset of $\R^{\omega_1}$ is $C$-embedded.
\end{enumerate}

Thus, Question~2 was answered before we posed it and the answer to Question~1
depends on one's assumptions: the Continuum Hypothesis implies the minimum
is~$\aleph_1$, and it is also consistent that it is larger than~$\aleph_1$.

The result from~\cite{MR4138431} can be viewed as a local version of
the main result of~\cite{MR2823691}: in a model obtained by adding supercompact 
many Random reals to a model of~$\CH$ every $C^*$-embedded subspace of every 
space of character less than~$\cee$ is $C$-embedded.
Indeed, one can create a model of $\arr>\aleph_1$ by  adding $\aleph_2$ or 
more Random reals to a model of~$\CH$.

In retrospect our paper~\cite{dhvmv2022} should have contained references
to~\cites{MR3003330,MR3239211,MR4138431} and we regret not finding these
references ourselves.
Nevertheless the methods and results of~\cite{dhvmv2022} and the present paper
are sufficiently different from the earlier ones that we feel they merit
publication.

In Sections~\ref{sec.three.copies} and~\ref{sec.connection} we give new 
examples and obtain topological and combinatorial translations of the statement
``$\R^{\omega_1}$ contains a closed copy of~$\N$ that is not $C^*$-embedded''
that suggest further interesting questions.

In Section~\ref{sec.three.copies} we present three constructions
of closed copies of~$\N$ that are not $C^*$-embedded in~$\R^{\omega_1}$:
one directly from an Aronszajn tree,
one directly from an Aronszajn continuum, and
one as the path space of an Aronszajn tree.
We decided to give all three examples because they show how versatile
these objects are.

In Section~\ref{sec.connection} we give the translations mentioned above
and give a fourth example that is of a somewhat different nature.

Section~\ref{sec.pseudo-A} deals with a class of topological spaces that
feature in the translations, and in Section~\ref{sec.Souslin} we present
models where $\CH$ fails but where the answer to Question~\ref{question.1}
is still~$\aleph_1$.


\section{Preliminaries}

By now the reader may have guessed that by ``a closed copy of~$\N$'' in some
space~$X$ we mean a closed subspace of~$X$ that is homeomorphic to the discrete
space~$\N$, in other words: a countably infinite closed and discrete subspace.

\medskip
In general we say that a subspace~$Y$ of a space~$X$ is $C$-embedded
if every continuous function $f:Y\to\R$ has a continuous extension to
all of~$X$.
If this holds for all \emph{bounded} continuous functions then we say that 
$Y$~is $C^*$-embedded in~$X$.

The way we shall show that a closed copy of~$\N$ is not $C^*$-embedded in~$X$
is by exhibiting disjoint subsets $A$ and $B$ of~$\N$ that are not
\emph{completely separated}, which means that whenever $g:X\to\R$ is 
bounded and continuous the closures of~$g[A]$ and~$g[B]$ intersect.
This then implies that the characteristic function of~$A$ has no continuous 
extension to~$X$.

\medskip
As mentioned in the introduction we shall uses Aronszajn trees and continua
in some of our constructions; Todor\v{c}evi\'c's article~\cite{MR776625} 
contains all the information that we need.

\medbreak
Below we shall be using a few `small' cardinals from~\cite{MR776622}.
These are $\pee$, $\dee$, and~$\uff$.

To define $\pee$ we first say that a family $\calA$ of subsets of~$\N$
has the \emph{strong finite intersection property} if for every finite
subfamily~$\calF$ of~$\calA$ the intersection~$\bigcap\calF$ is infinite. 
Next we call $P$ a \emph{pseudointersection} of~$\calA$ if $P$~is infinite
and $P\subseteq^* A$ for all~$A\in\calA$.
Then $\pee$~is the minimum cardinality of a family of subsets of~$\N$ with the
strong finite intersection property but without a pseudointersection.

The cardinal~$\dee$ is the minimum cardinality of a subset $D$ of $\N^\N$
with the property that for for every $g\in\N^\N$ there is an~$f\in D$
such that $g(n)\le f(n)$ for all~$n$.

The cardinal~$\uff$ is defined to be the minimum character of a free 
ultrafilter on~$\N$.

The cardinal~$\arr$ mentioned in the introduction is the minimum cardinality
of a family of subsets~$\calR$ of~$\N$ that behaves like an ultrafilter but 
for the finite intersection property: for every subset~$X$ of~$\N$ there is 
a member~$R$ of~$\calR$ such that $R\subseteq^*X$ or $R\cap X=^*\emptyset$.

It is relatively easy to show that $\pee$~is uncountable and that all three
cardinals are not larger than~$\cee$.
As a base for an ultrafilter is a defining family for both~$\pee$ and~$\arr$
we obtain $\pee,\arr\le\uff$, 
and \cite{MR776622}*{Theorem~3.1} shows $\pee\le\dee$.

\smallbreak
More information on these cardinal numbers can be found in~\cite{MR2768685}.

\medbreak
Any potentially unfamiliar topological notions will be defined when needed;
definitions not given here can be found in Engelking's book~\cite{MR1039321}.

\section{Closed copies of $\N$ that are not $C^*$-embedded}
\label{sec.three.copies}

This section contains further examples that show that the answer to the 
Question~\ref{question.2} is~$\aleph_1$.
We give three examples, based on Aronszajn trees and lines, 
of closed copies of~$\N$ that are not $C^*$-embedded in~$\Rwone$.
This may seem like overdoing things somewhat but we think that this 
presentation is more informative.

From our first two constructions we extract a few translations
of ``$\Rwone$ contains a closed copy of~$\N$ that is not $C^*$-embedded''
that allow us to construct a relatively simple third example and an even 
simpler fourth one.

\subsection{A closed copy of $\N$ that is not $C^*$-embedded,
            from an Aronszajn tree}
\label{subsec.Aronszajn.tree}

The first construction uses an Aronszajn tree to guide an embedding
of~$\N$ into~$\Rwone$.

The uncountability of~$\pee$ allows us to define a family
$\{A_t:t\in \finseq\}$ of infinite subsets of~$\N$ such that
\begin{enumerate}
\item if $s\subset t$ then $A_t\subset^* A_s$, and
\item for every $t$ the family $\{A_{t*n}:n\in\omega\}$ is a partition of~$A_s$.
\end{enumerate}

Next we take an Aronszajn subtree $T$ of $\finseq$.
Say an Aronszajn subtree of the set of all finite-to-one members of
$\finseq$, and such that $\{t*n:n\in\omega\}\subseteq T$ whenever $t\in T$.

For every non-zero $\alpha$ in~$\omega_1$ we let 
$\langle t(\alpha,n):n\in\omega\rangle$
enumerate the $\alpha$th level $T_\alpha$ of $T$ in a one-to-one fashion.
We abbreviate $A_{t(\alpha,n)}$ as~$A(\alpha,n)$.

By construction each of the families $\{A(\alpha,n):n\in\omega\}$ is pairwise
almost disjoint.
We can assume, after making finite modifications to the $A(\alpha,n)$,
that every family $\{A(\alpha,n):n\in\omega\}$ is in fact a partition of~$\N$.

\medbreak
We use the partitions to define a map $k\mapsto x_k$ from $\N$ to~$\Rwone$.

\smallbreak
First we set $x_k(0)=2^{-k}$ for all $k$.
This ensures that $X=\{x_k:k\in\N\}$ is a relatively discrete subspace 
of~$\Rwone$.

\smallskip
Second, for every non-zero $\alpha$ in $\omega_1$ we define
$$
x_{2k}(\alpha)=x_{2k+1}(\alpha)=m \text{ iff } k\in A(\alpha,m)
$$
This will ensure that $X$~is closed in~$\Rwone$
and that the sets $\{x_{2k}:k\in\N\}$ and $\{x_{2k+1}:k\in\N\}$ are not 
completely separated in~$\Rwone$.

\smallbreak
To see that $X$~is closed let $x\in\cl X$ and let $u$ be an ultrafilter
such that $x=\ulim x_k$.
We claim $u$ is in fact a fixed ultrafilter and hence that $x\in X$.

Since $u$~is a filter there is for every $\beta$ most one~$n$ such 
that $A(\beta,n)\in u$.
Let $B=\{\orpr\beta n: A(\beta,n)\in u\}$.
If $u$~were free then $A(\beta,n)\cap A(\gamma,m)$ would be infinite whenever
$\orpr\beta n,\orpr\gamma m\in B$.
By the construction of the family 
$\{A_t:t\in \finseq\}$ this would mean
that $\{t(\beta,n):\orpr \beta n\in B\}$ is linearly ordered in~$T$, and 
hence countable.

Take $\alpha$ such that
$T_\alpha\cap\{t(\beta,n):\orpr \beta n\in B\}=\emptyset$,
and let $m=\bigl\lceil x(\alpha)\bigr\rceil$.
Then $U=\N\setminus\bigcup_{i\le m}A(\alpha,i)$ belongs to~$u$,
and $x_{2k}(\alpha)=x_{2k+1}(\alpha)\ge x(\alpha)+1$ for all~$k\in U$.
This shows that $x(\alpha)\neq\ulim x_k(\alpha)$, which contradicts
the assumption that $x=\ulim x_k$.

\smallbreak
To see that $\{x_{2k}:k\in\N\}$ and $\{x_{2k+1}:k\in\N\}$ are not completely
separated in~$\Rwone$ let $g:\Rwone\to[0,1]$ be continuous.
It is well-known, see~\cite{MR1039321}*{Problem~2.7.12}, that
there are $\delta<\omega_1$ and a continuous function $h:\R^\delta\to[0,1]$
such that $g=h\circ\pi_\delta$.

Consider $A(\delta,0)$. 
By construction we know that for every non-zero $\alpha<\delta$
there is a single~$n_\alpha$ such $A(\delta,0)\subset^*A(\alpha,n_\alpha)$.
Let $x\in\R^\delta$ be given by $x(0)=0$ and $x(\alpha)=n_\alpha$,
then the subsequences $\langle x_{2k}:k\in A(\delta,0)\rangle$
and $\langle x_{2k+1}:k\in A(\delta,0)\rangle$ of $\langle x_k:k\in\N\rangle$
both converge to~$x$ and so $h(x)$~is in the closure of both
$\{g(x_{2k}):k\in\N\}$ and~$\{g(x_{2k+1}):k\in\N\}$.

\subsection{Another closed copy of $\N$ that is not $C^*$-embedded,
            from an Aronszajn line}
\label{subsec:another.copy}

Let $L$ be an Aronszajn continuum: 
a first-countable linearly ordered continuum of weight~$\aleph_1$ with 
the property that the closure of every countable set is second-countable,
see~\cite{MR776625}*{Section~5}.
We can also assume, without loss of generality, that $L$~has no non-trivial
separable intervals.

Let $\omegaoneseq x$ enumerate a dense subset of~$L$, where we assume that
$x_0=\min L$ and $x_1=\max L$.
Using the first-countability of~$L$ we find that 
$L=\bigcup_{\alpha<\omega_1}\cl\{x_\beta:\beta\le\alpha\}$,
that is, $L$ is the union of an increasing sequence of second-countable
compact subsets.
Upon thinning out the sequence we obtain a strictly increasing
sequence~$\omegaoneseq K$ of second-countable
compact subsets whose union is equal to~$L$.
The assumption on the intervals of~$L$ implies that each~$K_\alpha$ is 
nowhere dense.

\medbreak
We claim that every $K_\alpha$ is a $G_\delta$-set of~$L$.
By the first-countability of~$L$ this is clear if $\alpha$~is finite,
so we assume below that $\alpha$~is infinite, and hence that 
$\min L$ and $\max L$ belong to~$K_\alpha$.  

\smallbreak
Since $K_\alpha$ is second-countable we can find a countable family~$\calI$ of 
open intervals in~$L$ such that $\{I\cap K_\alpha:I\in\calI\}$ is a base
for the topology of~$K_\alpha$.

Every convex component~$C$ of $L\setminus K_\alpha$ is of the form $(a_C,b_C)$,
with $a_C,b_C\in K_\alpha$.
If $C$ and $D$ are two such components then $b_C<a_D$ or $b_D<a_C$.
For each $C$ take $I_C\in\calI$ such 
that $I_C\cap K_\alpha\subseteq[b_C,\max L]$.
Then $a_C\notin I_C$ and so $b_D\notin I_C$ whenever $b_D<a_C$.
It follows that $I_C\neq I_D$ whenever $C\neq D$.
This shows that there are at most countably many convex components
in the complement of~$K_\alpha$.

\smallbreak
Enumerate these components as $\omegaseq C$ and choose for every $n\in\omega$
sequences $\langle a(n,k):k\in\omega\rangle$ 
and $\langle b(n,k):k\in\omega\rangle$ in~$C_n$ such that
$a(n,k)\downarrow a_{C_n}$ and $b(n,k)\uparrow b_{C_n}$.

Then $C_n=\bigcup_{k\in\omega}[a(n,k),b(n,k)]$ for all~$n$.
Define $F_k=\bigcup_{n\le k}[a(n,k),b(n,k)]$ for all~$k$.
Then $\omegaseq[k]F$ is a sequence of closed sets and its union 
is equal to the complement of~$K_\alpha$.

\medbreak
Since $L$~has weight~$\aleph_1$ there is a compactification~$\gamma\N$
of~$\N$ such that $\gamma\N\setminus\N$ is (homeomorphic to)~$L$,
see~\cite{MR1039321}*{Problem~3.12.18\,(c)}.
Take the quotient of $\gamma\N\times\{0,1\}$ obtained by identifying
$\orpr x0$ and $\orpr x1$ for all~$x\in L$.

The result is a new compactification $\delta\N$ of~$\N$ with remainder
equal to~$L$ and in which $\N$~is the union of two subsets~$A$ and~$B$
such that $L=\cl A\cap\cl B$.

We map $\delta\N$ into $[0,1]^{\omega_1}$ in such a way that the image of~$\N$
will be a closed subset of~$(0,1)^{\omega_1}$ that is not $C^*$-embedded.

For every $\alpha\ge1$ we let $f_\alpha:\delta\N\to[0,1]$ be continuous
such that $K_\alpha=f_\alpha\preim(0)$ and $f_\alpha[\N]\subseteq(0,1)$.
We let $f_0:\delta\N\to[0,1]$ be the continuous map determined 
by $f_0(k)=\frac12+2^{-k-2}$;
it maps $L$ to~$\{\frac12\}$ and $\N$ into~$(\frac12,1)$.

The diagonal map $F$ of $\omegaoneseq f$ maps $\delta\N$ to~$[0,1]^{\omega_1}$
and maps $\N$ into~$(0,1)^{\omega_1}$.

The first coordinate~$f_0$ ensures that $F[\N]$ is relatively discrete 
in~$(0,1)^{\omega_1}$;
it remains to show that it is closed and not $C^*$-embedded.

\smallbreak
To see that $F[\N]$ is closed in~$(0,1)^{\omega_1}$ observe that for 
every $x\in L$ there is an~$\alpha$ such that $x\in K_\alpha$;
but then $f_\beta(x)=0$ for $\beta\ge\alpha$.
It follows that $F[\N]=F[\delta\N]\cap(0,1)^{\omega_1}$.

\smallbreak
To see that $F[\N]$ is not $C^*$-embedded in $(0,1)^{\omega_1}$
let $g:(0,1)^{\omega_1}\to[0,1]$ be continuous. 
We show that the closures of~$g[A]$ and~$g[B]$ intersect.

As above there is an $\alpha$ such that $g$~factors through the first $\alpha$
coordinates, that is, there is a continuous map $h:(0,1)^\alpha\to[0,1]$
such that $g=h\circ\pi_\alpha$.
Take $x\in L\setminus K_\alpha$.
Then $x\in\cl A\cap\cl B$, 
hence $\pi_\alpha(x)\in\cl(\pi_\alpha[A])\cap\cl(\pi_\alpha[B])$.
But because $x\notin K_\beta$ for all $\beta\le\alpha$ we find 
that $\pi_\alpha(x)\in(0,1)^\alpha$ and hence we conclude
that $h(\pi_\alpha(x))\in\cl(g[A])\cap\cl(g[B])$.

\subsection{A characterization}

From the foregoing example we extract a characterization of there being
a closed copy of~$\N$ in~$\Rwone$ that is not $C^*$-embedded.

\begin{theorem}\label{thm.char}
The following three statements are equivalent:
\begin{enumerate}
\item There is closed copy of~$\N$ in~$\Rwone$ that is not
      $C^*$-embedded.
\item There is closed copy of~$\N$ in~$\Rwone$ that is not
      $C$-embedded.
\item There is a compact space~$X$ with a cover consisting of $\aleph_1$~many
      zero-sets that has no countable subcover.
\end{enumerate}
\end{theorem}

\begin{proof}
That (1) implies~(2) is clear.

To prove (2) implies (3) we take a countable closed and discrete subset~$N$
of~$(0,1)^{\omega_1}$ that is not $C$-embedded.
Let $K=\cl N\setminus N$, where we take the closure in~$[0,1]^{\omega_1}$.
For every $\alpha\in\omega_1$ and $i\in\{0,1\}$ we let 
$A(\alpha,i)=\{x\in K:x_\alpha=i\}$.
Then $\{A(\alpha,i):\orpr\alpha i\in\omega_1\times2\}$ is a cover of~$K$
by $\aleph_1$~many $G_\delta$-sets.
We show that there is no~$\alpha\in\omega_1$ such that 
$\{A(\beta,i):\orpr\beta i\in\alpha\times2\}$ covers~$K$.

Let $\alpha\in\omega_1$; we can assume that the projection 
$\pi_\alpha:[0,1]^{\omega_1}\to[0,1]^\alpha$ is one-to-one on~$N$.
If $\{A(\beta,i):\orpr\beta i\in\alpha\times2\}$ covers~$K$ then for every
$x\in K$ there is a $\beta\in\alpha$ such that $x_\beta\in\{0,1\}$,
and hence $\pi_\alpha(x)\notin(0,1)^\alpha$.
We see that $\pi_\alpha[K]$ is disjoint from~$(0,1)^\alpha$ and hence
that $\pi_\alpha[N]$~is closed in~$(0,1)^\alpha$ and hence also $C$-embedded
because $(0,1)^\alpha$~is metrizable.
But then $N$~would be $C$-embedded in~$(0,1)^{\omega_1}$.  

\medbreak
To prove that (3) implies (1) we proceed as in 
Section~\ref{subsec:another.copy}.
Let $X$ be a space as in~(3) and let $\{A_\alpha:\alpha\in\omega_1\}$
be the cover by~zero-sets without a countable subcover.
We may assume that $X$~has weight~$\aleph_1$, for example, by choosing
a sequence $\omegaoneseq f$ of continuous functions from~$X$ to~$[0,1]$
such that $A_\alpha=f_\alpha\preim(0)$ for all~$\alpha$.
The image~$K$ of~$X$ under the diagonal map of the sequence has the same
property as $X$ itself, where $B_\alpha=\{x\in K:x_\alpha=0\}$ defines
the family of zero-sets.

The construction in Section~\ref{subsec:another.copy} now yields a closed
copy of~$\N$ in~$(0,1)^{\omega_1}$ that is not $C^*$-embedded.     
\end{proof}

\begin{remark}
Of course $2^2=4$ is also an equivalent of statement~(1), as both are true,
but this theorem should be understood as a translation:
to construct the desired embedding it is necessary and sufficient to construct
a particular type of compact topological space.
\end{remark}

\begin{remark}
It is interesting to see that the formally weaker statement~(2) implies
statement~(1); what is hidden in the proof is that from the copy that
is not $C$-embedded one constructs a copy that is not $C^*$-embedded
by taking its closure in~$[0,1]^{\omega_1}$, doubling the resulting 
compactification, then glueing the remainders onto each other and
find a suitable embedding of the resulting space.
\end{remark}

\subsection{Yet another closed copy of $\N$ that is not $C^*$-embedded,
            from an Aronszajn tree}

To see an application of Theorem~\ref{thm.char} we create yet another
closed copy of~$\N$ in $\Rwone$ that is not $C^*$-embedded, by 
exhibiting a space that satisfies the properties in~(3) in the theorem.

We let $T$ be an Aronszajn tree and we take its \emph{path space} $\sigma T$,
where a path is a linearly ordered subset~$P$ that is also an initial segment: 
if $t\in P$ and $s\le t$ then $s\in P$.
We view $\sigma T$, via characteristic functions, as a subspace of the 
Cantor cube~$\{0,1\}^T$.
For more on this construction see~\cite{MR1355064}.

The condition that the paths be linearly ordered ensures that $\sigma T$~is 
closed and hence compact.
The weight of~$\sigma T$ is at most that of~$\{0,1\}^T$, that is $\aleph_1$.

For $\alpha\in\omega_1$ we let $K_\alpha$ be the set of paths that are
of length less than~$\alpha$.
To see that $K_\alpha$ is closed note that $p\in K_\alpha$ 
iff $p\cap T_\alpha=\emptyset$.
That is $K_\alpha=\sigma T\setminus \bigcup_{t\in T_\alpha}O_t$,
where $O_t=\{p:t\in p\}$.
The sets $O_t$ are clopen, so the union $\bigcup_{t\in T_\alpha}O_t$ is an 
open $F_\sigma$-set.

Because $T$ is uncountable no countable subfamily 
of $\{K_\alpha:\alpha\in\omega_1\}$ covers $\sigma T$.

Note that, as every path is countable, the space~$\sigma T$ is actually 
Corson-compact.

\section{The connection with Aronszajn trees and lines}
\label{sec.connection}

Each of the three constructions in the previous section uses an Aronszajn tree
or line as input.
The following theorem, which adds three more statements to the list
in Theorem~\ref{thm.char}, makes precise how these structures enter the
constructions.

\begin{theorem}
The following statements are equivalent. \label{thm.char.N}
\begin{enumerate}
\item There is closed copy of~$\N$ in~$\Rwone$ that is not
      $C^*$-embedded.
\item There is closed copy of~$\N$ in~$\Rwone$ that is not
      $C$-embedded.
\item There is a compact space~$X$ with a cover consisting of $\aleph_1$~many
      zero-sets that has no countable subcover.
\item There is a compact space~$X$ of weight~$\aleph_1$ with a cover consisting
      of $\aleph_1$~many zero-sets that has no countable subcover.
\item The space~$\N^*$ has a cover by $\aleph_1$~many zero-sets that has no 
      countable subcover.
\item There is an $\omega_1\times\omega$-matrix \label{item.matrix}
      $\langle A(\alpha,n):\orpr\alpha n\in\omega_1\times\omega\rangle$
      of infinite subsets of~$\N$ such that
      \begin{enumerate}
      \item for every countable subset $C$ of $\omega_1$ there is a function
            $f:C\to\omega$ such that $\{A(\alpha,f(\alpha)):\alpha\in F\}$
            has the strong finite intersection property, and
      \item there is \emph{no function} $f:\omega_1\to\omega$ such that
            $\{A(\alpha,f(\alpha)):\alpha\in\omega_1\}$ has the strong finite
            intersection property.
      \end{enumerate}
\end{enumerate}
\end{theorem}

\begin{proof}
Theorem~\ref{thm.char} established the equivalence of (1), (2), and~(3).
In the proof that (3) implies~(1) we proved implicitly that
(3) implies~(4) and (4) implies~(1).

Clearly (5) implies (3).

To prove that (4) implies (5) we take a continuous map $f$ from~$\N^*$
onto~$X$ and take the preimages of the members of the given cover.
This yields the desired cover of~$\N^*$.

It remains to show that (5) and (6) are equivalent.
This follows from the strong zero-dimensionality of~$\N^*$:
if $Z$~is a zero-set in~$\N^*$ then one can cover $\N^*\setminus Z$ by a 
countable pairwise disjoint family of clopen sets.
This family can be expressed as $\{A_n^*:n\in\omega\}$, where each $A_n$~is 
an infinite subset of~$\N$.

Conversely if $\{A_n:n\in\omega\}$ is a family of infinite subsets of~$\N$
then $\N^*\setminus\bigcup_{n\in\omega}A_n^*$ is a zero-set.

Thus a family $\{Z_\alpha:\alpha\in\omega_1\}$ of zero-sets of~$\N^*$ can be
represented by a matrix 
$\langle A(\alpha,n):\orpr\alpha n\in\omega_1\times\omega\rangle$
of infinite subsets of~$\N$ such that 
$Z_\alpha=\N^*\setminus\bigcup_{n\in\omega}A(\alpha,n)^*$.

Then condition~(a) expresses that no countable subfamily covers $\N^*$,
and condition~(b) expresses that the family does cover~$\N^*$.
\end{proof}

The matrix $\langle A(\alpha,n):\orpr\alpha n\in\omega_1\times\omega\rangle$
of sets from Section~\ref{subsec.Aronszajn.tree}, that resulted from 
enumerating the levels of the Aronszajn tree as 
$\langle t(\alpha,n):n\in\omega\rangle$,
satisfies the conditions in item~\eqref{item.matrix} of 
Theorem~\ref{thm.char.N}.

It would seem natural to call such a matrix an Aronszajn matrix and
a compact space with a cover of cardinality~$\aleph_1$ by closed
$G_\delta$-sets without a countable subcover an Aronszajn compactum.
This usage would conflict with that of Hart and Kunen in~\cite{MR2516236};
and, more importantly, it would not be quite correct, as we show next.

\subsection{A matrix and space that are not derived from an Aronszajn tree}

The three examples constructed in Section~\ref{sec.three.copies} all have in 
common that they have an \emph{increasing} cover of length~$\omega_1$ by 
closed $G_\delta$-sets.

Here we construct a compact space of weight~$\aleph_1$ with an $\aleph_1$-sized
cover by closed $G_\delta$-sets that has no countable subcover, and
that is definitely not increasing.
The space is a variation of Example~7 in~\cite{MR584666}.

To begin we take an injective map $f:\omega_1\to\R$ with the property that
for every~$\alpha$ the image of the interval 
$I_\alpha=\bigl[\omega\cdot\alpha,\omega\cdot(\alpha+1)\bigr)$ under~$f$ is 
dense in~$\R$.
This is easily arranged, for example by taking $\aleph_1$~many cosets
of the subgroup of rationals and mapping each interval~$I_\alpha$ onto one of 
these cosets.

We let $X$ be the set of all subsets of~$\omega_1$ on which $f$~is 
monotonically increasing; we identify $X$, via characteristic functions,
with a subset of~$2^{\omega_1}$ and give it the subspace topology.

The complement of~$X$ is open: if $x\notin X$ then there are two ordinals
$\alpha$ and $\beta$ such that $x_\alpha=x_\beta=1$, $\alpha\in\beta$,
and $f(\beta)<f(\alpha)$.
Then $\{y:y_\alpha=y_\beta=1\}$ is an open set disjoint from~$X$.
It follows that $X$~is compact.

As subsets of $\R$ that are well-ordered by the normal order are countable
the space~$X$ is Corson compact.

It remains to exhibit a cover of~$X$ by closed $G_\delta$-sets that has no
countable subcover.

To this end we let $G_\alpha=\{x\in X:(\forall\beta\in I_\alpha)(x_\beta=0)\}$.
This is a closed $G_\delta$-set; it is the intersection of countably many
basic clopen sets: $G_\alpha=\bigcap_{\beta\in I_\alpha}\{x:x_\beta=0\}$.

To see that $\{G_\alpha:\alpha\in\omega_1\}$ is a cover of~$X$, let $x\in X$.
Then, because $S=\{\beta:x_\beta=1\}$ is countable, there is an~$\alpha$
such that $S\subset\alpha$; then $S\cap I_\alpha=\emptyset$ and 
so $x\in G_\alpha$.

To see that no countable subfamily covers $X$ we let $\delta\in\omega_1$.
We take a subset~$A$ of~$\R$ that is ordered in order-type $\delta+2$ by the
normal order of~$\R$ and so we list~$A$ 
as $\langle a_\alpha:\alpha<\delta+2\rangle$ in increasing order.
Next take a sequence $\langle \gamma_\alpha:\alpha<\delta+2\rangle$ of ordinals
such that $\gamma_\alpha\in I_\alpha$ and $a_\alpha<f(\gamma_\alpha)<a_{\alpha+1}$
for all~$\alpha$.
Then the set $\{\gamma_\alpha:\alpha<\delta+2\}$ 
determines a point in~$X$ that is not in~$\bigcup_{\alpha\in\delta}G_\alpha$.

The same argument enables one to show that the sets $G_\alpha$ are quite
independent: given two disjoint countable sets of ordinals $A$ and $B$
one can find points 
in $\bigcap_{\alpha\in A}G_\alpha\setminus\bigcup_{\beta\in B}G_\beta$.

Via a map from $\N^*$ onto~$X$ we can then create a matrix that is quite 
different from the ones derived from Aronszajn trees.

\section{Pseudo-Aronszajn compacta}
\label{sec.pseudo-A}

Let us, for the nonce, call a compact space a \emph{pseudo-Aronszajn compactum} 
if it has a cover of cardinality~$\aleph_1$ by closed $G_\delta$-sets
that has no countable subcover.
We let $\calA$ denote the class of these compacta. 

It is readily seen that $\calA$ is closed under taking (compact) preimages:
simply pull back the cover.

We have established that every Aronszajn continuum is in $\calA$, and 
hence that a Souslin continuum is a ccc compactum in~$\calA$.

The ordinal space $\omega_1+1$ does not belong to~$\calA$ as every 
$G_\delta$-set that contains the point~$\omega_1$ is co-countable.

Somewhat surprisingly, 
uncountable compact metrizable spaces may or may not all be 
pseudo-Aronszajn compacta.
They all are under~$\CH$ and they all are not under $\MA+\neg\CH$.

\begin{proposition}[$\CH$]\label{eerste}
If $X$ is compact and admits a continuous map $f:X\to\R$ such 
that $f[X]$ is uncountable, then $X\in\calA$.
\end{proposition}
\begin{proof}
The image $f[X]$ is in $\calA$, 
as witnessed by the family of singleton subsets.  
\end{proof}

\begin{proposition}[$\MA+\neg\CH$]\label{tweede}
If $X$ is compact, uncountable and hereditarily Lindel\"of, 
then $X\not\in \calA$.
\end{proposition}

\begin{proof}
Let $\calZ$ be a witness of the fact that the uncountable compact 
hereditarily space~$X$ is in~$\calA$. 
We will derive a contradiction. 

Let $X_0 = X$ and  
$U_0 = \bigcup_{Z\in \calZ} \Int_{X_0} Z$. 
There is a countable subfamily~$\calZ_0$ of~$\calZ$ such that 
$U_0 = \bigcup_{Z\in {\calZ_0}} \Int_{X_0} Z$.

Assume that for some $\alpha<\omega_1$, we defined closed sets~$X_\beta$, 
open sets~$U_\beta$, and subfamilies~$\calZ_\beta$ of~$\calZ$,
for all $\beta < \alpha$.

Let $V=\bigcup_{\beta < \alpha} U_\beta$, 
$X_\alpha = (\bigcap_{\beta < \alpha} X_\beta)\setminus V$, 
and $\calS= \bigcup_{\beta < \alpha} \calZ_\beta$. 

Inside $X_\alpha$ let $W = \bigcup_{Z\in \calZ} \Int_{X_\alpha} (Z\cap X_\alpha)$.
Then $U_\alpha = V\cup W$ is open in~$X$, and there is a countable 
subcollection~$\calT$ of~$\calZ$ such that 
$W = \bigcup_{Z\in \calT} \Int_{X_\alpha} (Z\cap X_\alpha)$. 
We let $\calZ_\alpha = \calS\cup \calT$. 

There is a first $\alpha\in\omega_1$ such that $U_\alpha = U_{\alpha + 1}$. 
If $Y=X\setminus U_\alpha$ is countable, then we are clearly done. 
If $Y$ is uncountable, then for every $Z\in \calZ$, 
the intersection $Z\cap Y$ is nowhere dense in~$Y$. 
But this contradicts $\MAnegCH$, for $Y$~is an uncountable compact ccc space 
with a cover by fewer than~$\cee$ many nowhere dense sets.
\end{proof}

One may wonder whether $\MAnegCH$ prevents more compact spaces from being
pseudo-Aronszajn. 
We have seen that a Souslin line is a pseudo-Aronszajn compactum and 
we also know that $\MAnegCH$ implies there are no Souslin lines.
Thus we may conjecture that it implies that there are no pseudo-Aronszajn
compacta that are ccc.

However, as there are pseudo-Aronszajn compacta of weight~$\aleph_1$
one can construct a compactification~$\gamma\N$ of~$\N$ with a 
pseudo-Aronszajn remainder.
That compactification is itself also pseudo-Aronszajn: simply add the isolated
points to the cover of the remainder.
Thus we see that $\calA$~contains separable spaces.

We can strengthen the ccc assumption by making it hereditary; it is well known
that having the hereditary ccc is equivalent to every relatively 
discrete subspace being countable, see~\cite{MR1039321}*{Problem~2.7.9(b)}
for example.
Thus, the hereditary ccc is also a weakening of the hereditary Lindel\"of 
property and a positive answer to the following question would yield a 
strengthening of Proposition~\ref{tweede}.

\begin{question}
Does $\MAnegCH$ imply that uncountable compact hereditarily ccc spaces 
are not pseudo-Aronszajn?
\end{question}

We remark in passing that it is also unknown whether compact hereditarily ccc
spaces are continuous images of~$\N^*$, see~\cite{hvm2022}*{Question~44}.

\section{$\notCH$ and a closed copy of~$\N$ that is $C^*$-embedded but not $C$-embedded}
\label{sec.Souslin}

In section~\ref{sec.three.copies} we used an Aronszajn tree to guide an 
embedding of~$\N$ into~$\R^{\omega_1}$ so as to obtain a closed copy of~$\N$ that
is not $C^*$-embedded.
In this section we use an Aronszajn tree again, this time to create closed 
copies of~$\N$ in~$\R^{\omega_1}$ that are $C^*$-embedded but not $C$-embedded, 
in models where $\CH$ fails.
Thus we see that it is consistent with~$\neg\CH$ that the answer to 
Question~\ref{question.1} be~$\aleph_1$.

The embedding will be much like the one from an arbitrary Aronszajn tree but
with a few changes.
We shall show that the following assumptions suffices to create a closed copy
of~$\N$ in~$\R^{\omega_1}$ that is $C^*$-embedded but not $C$-embedded.

\begin{assumption}
There are an Aronszajn tree~$S$ and a family $\{A_s:s\in S\}$ of infinite
subsets of~$\N$ such that 
\begin{itemize}
\item if $s<t$ then $A_t\subset^* A_s$, and
\item if $Y\subseteq\N$ then there is an ordinal~$\alpha$ in~$\omega_1$
      such that for every $s\in S_\alpha$ either $A_s\subseteq^*Y$
      or $A_s\cap Y=^*\emptyset$.
\end{itemize}
Here $S_\alpha$ denotes the $\alpha$th level of~$S$.
We also assume that every level~$S_\alpha$, except~$S_0$, is infinite
and that every node in~$S$ has infinitely many direct successors.

In addition we make finite modifications to each~$A_s$ so that
$\{A_s:s\in S_\alpha\}$ is a partition of~$\N$.
\end{assumption}

\subsection{The construction}

We shall embed $\N$ into the following product:
$$
\Pi=C\times\prod_{1\le\alpha<\omega_1}S_\alpha
$$
where $C$ is the subspace $\{0\}\cup\{2^{-n}:n\in\N\}$ of~$\R$ and each other
factor~$S_\alpha$ has the discrete topology.
This product is homeomorphic to the product $C\times\N^{\omega_1}$,
which in turn can be embedded as a $C$-embedded subspace into~$\R^{\omega_1}$.

\smallskip
Now we are ready to define the embedding.

To begin we set $x_k(0)=2^{-k}$ for all~$k$; this ensures that the image
will be relatively discrete.

If $\alpha\in[1,\omega_1)$ then we set $x_k(\alpha)=s$ iff $k\in A_s$
(and $s\in S_\alpha$ of course).
 
This defines our copy $N=\{x_k:k\in\N\}$ of~$\N$ in~$\Pi$.

\subsubsection*{$N$ is closed in $\Pi$}

Let $v\in\Pi$.
Then $\langle v_\alpha:1\le\alpha<\omega_1\rangle$ is a sequence in~$S$
with $xv_\alpha\in S_\alpha$ for all~$\alpha$.

As $S$ is Aronszajn there are $\alpha$ and $\beta$ with $\alpha<\beta$ and
such that $v_\alpha$ and $v_\beta$ are incomparable.
Let $w$ be the predecessor of~$v_\beta$ in~$S_\alpha$.
Then $A_w\cap A_{v_\alpha}=\emptyset$ and so, because $A_{v_\beta}\subset^* A_w$
the intersection $A_{v_\beta}\cap A_{v_\alpha}$ is finite.

Let $U$ be the basic neighbourhood 
$\{x\in\Pi:x_\alpha=v_\alpha$ and $x_\beta=v_\beta\}$ of~$v$.
Then $x_k\in U$ iff $k\in A_{v_\beta}\cap A_{v_\alpha}$, hence $U\cap N$ is finite.

We see that $N$ is a locally finite and relatively discrete subset of~$\Pi$,
hence $N$~is closed and discrete.

\subsubsection*{$N$ is $C^*$-embedded in $\Pi$}

Let $Y\subseteq\N$; we show that the sets $\{x_k:k\in Y\}$ 
and $\{x_k:k\notin Y\}$ are completely separated in~$\Pi$.

Let $\alpha$ be such that $A_s\subseteq^*Y$ or $A_s\subseteq^*\N\setminus Y$ 
for all~$s\in S_\alpha$ and divide $S_\alpha$ into two sets:
$I=\{s\in S_\alpha:A_s\subseteq^*Y\}$ and
$J=\{s\in S_\alpha:A_s\cap Y=^*\emptyset\}$.

In this way we create four subsets of~$\N$:
\begin{enumerate}
\item $Y_1=\bigcup\{A_s\cap Y:s\in I\}$,
\item $Y_2=\bigcup\{A_s\cap Y:s\in J\}$,
\item $Z_1=\bigcup\{A_s\setminus Y:s\in J\}$, and
\item $Z_2=\bigcup\{A_s\setminus Y:s\in I\}$.
\end{enumerate}

To begin we observe that $Y_2\cup Z_2$ intersects every $A_s$ in a finite set.
Because $\{A_s:s\in S_\alpha\}$ is a partition of~$\N$ this implies, as in the
proof that $N$~is closed, that 
$D=\{x_k\restr(\alpha+1):k\in Y_2\cup Z_2\}$ is
a closed and discrete subset of the 
subproduct~$\Pi_\alpha=C\times\prod_{1\le\beta\le\alpha}S_\beta$.
This product is separable and metrizable, hence $D$~is $C$-embedded in this 
subproduct, this implies that in particular, 
$\{x_k\restr(\alpha+1):k\in Y_2\}$ and $\{x_k\restr(\alpha+1):k\in Z_2\}$ 
are completely separated in~$\Pi_\alpha$.

Furthermore, because $N$~is relatively discrete in the subproduct the
set $D$ is disjoint from the closure of 
$\{x_k\restr(\alpha+1):k\in Y_1\cup Z_1\}$.

Finally the $\alpha$th coordinates of the $x_k$ and ensure that
$\{x_k(\alpha):k\in Y_1\}$ and $\{x_k(\alpha):k\in Z_1\}$ are disjoint.
And because $S_\alpha$~has the discrete topology this shows that
$\{x_k\restr(\alpha+1):k\in Y_1\}$ and $\{x_k\restr(\alpha+1):k\in Z_1\}$
are completely separated in~$\Pi_\alpha$.

We conclude that 
$\{x_k\restr(\alpha+1):k\in Y\}$ and $\{x_k\restr(\alpha+1):k\notin Y\}$
are completely separated in~$\Pi_\alpha$.

\subsubsection*{$N$ is not $C$-embedded in $\Pi$}

We show that the function $f:N\to\R$ that maps $x_k$ to $k$ has no continuous
extension to~$\Pi$.

Assume $g:\Pi\to\R$ is continuous and such that $g(x_k)=k$ for all~$k$.
As before we can factor~$g$ through a partial product: 
there are a $\delta$ and a continuous function 
$h:C\times\prod_{1\le\alpha<\delta}S_\alpha$
such that $g=h\circ\pi_\delta$.

Let $s\in S_\delta$ and let $s_\alpha$ denote its predecessor in~$S_\alpha$,
for $\alpha\in[1,\delta)$. 
Take such an~$\alpha$, then by construction $A_s\subseteq^*A_{s_\alpha}$ 
and so $x_k(\alpha)=s_\alpha$ for all but finitely many~$k\in A_s$.

Because $A_s$ is infinite this implies that the point~$v$, with $v(0)=0$
and $v(\alpha)=s_\alpha$ for $\alpha\in[1,\delta)$,
is an accumulation point of $\{\pi_\delta(x_k):k\in A_s\}$ and hence
that $h(v)>k$ for all~$k$, a contradiction.

\subsection{A model}

To finish we show that our assumption is actually consistent with the
negation of~$\CH$.
Chapters~VII and~VIII of~\cite{MR597342} provide all the forcing background
that we need.

We let $S$ be an Aronszajn tree as constructed 
in~\cite{MR597342}*{Theorem~II.5.9}.
This tree is a subtree of the subtree~$T$ of $\omega^{<\omega_1}$ that consists
of all finite-to-one sequences of natural numbers
and it has the property that for every $s\in S$ the set of direct
successors is $\{s*n:n\in\omega\}$.
This tree has the advantage that if a partial order preserves~$\omega_1$
then it will add not an $\omega_1$-branch to it, as such a branch would give a 
finite-to-one map from~$\omega_1$ to~$\omega$.

Next we work Exercise~VIII\,(A10) in~\cite{MR597342}, that is, we perform 
an $\omega_1$~long finite support iteration of $\sigma$-centered partial 
create an ultrafilter on~$\N$ of character~$\aleph_1$.

More explicitly: we form a sequence $\langle M_\alpha:\alpha\le\omega_1\rangle$
of models, together with sequences $\omegaoneseq u$ and $\omegaoneseq U$.
Together these satisfy
\begin{enumerate}
\item $u_\alpha$ is an ultrafilter on~$\N$ in $M_\alpha$,
\item $M_{\alpha+1}$ is obtained by forcing over $M_\alpha$ with the partial 
      order~$\E(u_\alpha)$ descibed below, which produces a subset~$U_\alpha$
      of~$\N$ such that $U_\alpha\subseteq^*X$ for all $X\in u_\alpha$, and
\item $u_{\alpha+1}$ extends $u_\alpha\cup\{U_\alpha\}$.
\end{enumerate}
 
For a free ultrafilter~$u$ on~$\N$ we define the partial order
$$
\E(u)=\{\orpr sU: s\in[\N]^{<\omega}, U\in u\}
$$
ordered by $\orpr sU\le\orpr tV$ iff
\begin{itemize}
\item $t\subseteq s$, 
\item $U\subseteq V$, and 
\item $s\setminus t\subseteq V$.
\end{itemize}
If $G$ is a generic filter on~$\E(u)$ then 
$E=\bigcup\{s:(\exists U\in u)(\orpr sU\in G)\}$ is an infinite
subset of~$\omega$ such that $E\subseteq^* U$ for all~$U\in u$.

\subsubsection*{The assumption}

The iteration yields a ccc partial order with a dense subset of 
cardinality~$\cee$.
Therefore it preserves all cardinal arithmetic from the ground model~$M_0$.
Thus $M_{\omega_1}$~can be made to satisfy any consistent cardinal arithmetic,
in particular $2^{\aleph_0}$ can have any value it ought to have.

We define a family $\{A_s:s\in S\}$ of infinite subsets as in our assumption.
We start by setting $A_\emptyset=\N$.

For the successor steps we fix a definable bijection $f:\N^2\to\N$,
say $f(m,n)=\frac12(m+n)(m+n+1)+m$.

Going from $\alpha$ to $\alpha+1$ we assume that $\{A_s:s\in S_\alpha\}$
is in~$M_\alpha$ and build $\{A_t:t\in S_{\alpha+1}\}$ in~$M_{\alpha+1}$.
We take for every~$s\in S_\alpha$ the counting function $c_s:\N\to A_s$;
these functions belong to~$M_\alpha$.
For every $s\in S_\alpha$ and $n\in\N$ we define 
$A_{s*n}=f_s\bigl[f[\{n\}\times U_\alpha]\bigr]$.
In words: we use $f_s\circ f$ to create a partition of~$A_s$ in~$M_\alpha$
and then copy~$U_\alpha$ to each element of that partition by maps 
in~$M_\alpha$.

In this way we ensure that each $A_{s*n}$ has the property that $U_\alpha$ has:
for every subset~$Y$ of~$\N$ that is in~$M_\alpha$ we have 
$A_{s*n}\subseteq Y$ or $A_{s*n}\cap Y=^*\emptyset$.
The resulting family $\{A_t:t\in S_{\alpha+1}\}$ is defined from $U_\alpha$ 
and members of~$M_\alpha$, hence it is in~$M_{\alpha+1}$.

In case $\alpha\in\omega_1$ is a limit the partial family 
$\{A_s:s\in\bigcup_{\beta\in\alpha}S_\beta\}$ belongs to~$M_\alpha$.
So in~$M_\alpha$ we can find a family $\{A_t:t\in S_\alpha\}$ of infinite
subsets of~$\N$ such that $A_t\subseteq^*A_s$ whenever $s<t$.
 
\smallskip
To see that the resulting family has the second property in our assumption
we let~$Y$, in $M_{\omega_1}$, be a subset of~$\N$.
By well-known properties of finite-support iterations of ccc partial orders
there is an~$\alpha\in\omega_1$ such that $Y\in M_\alpha$.
But then for all $s\in S_{\alpha+1}$ we have $A_s\subseteq Y$ 
or $A_s\cap Y=^*\emptyset$.


\begin{bibdiv}
  \begin{biblist}



\bib{MR584666}{article}{
   author={Alster, K.},
   author={Pol, R.},
   title={On function spaces of compact subspaces of $\Sigma $-products of
   the real line},
   journal={Fund. Math.},
   volume={107},
   date={1980},
   number={2},
   pages={135--143},
   issn={0016-2736},
   review={\MR{584666}},
   doi={10.4064/fm-107-2-135-143},
   url={https://eudml.org/doc/211111},
}
		

\bib{MR2823691}{article}{
   author={Barman, Doyel},
   author={Dow, Alan},
   author={Pichardo-Mendoza, Roberto},
   title={Complete separation in the random and Cohen models},
   journal={Topology Appl.},
   volume={158},
   date={2011},
   number={14},
   pages={1795--1801},
   issn={0166-8641},
   review={\MR{2823691}},
   doi={10.1016/j.topol.2011.06.014},
}

\bib{MR2768685}{article}{
   author={Blass, Andreas},
   title={Combinatorial cardinal characteristics of the continuum},
   book={
       title={Handbook of set theory. Vols.~1, 2, 3},
       editor={Foreman, Matthew},
       editor={Kanamori, Akihiro},
       publisher={Springer, Dordrecht},
       isbn={978-1-4020-4843-2},
       date={2010},   
   },
   pages={395--489},
   review={\MR{2768685}},
   doi={10.1007/978-1-4020-5764-9\_7},
}


\bib{MR776622}{article}{
   author={van Douwen, Eric K.},
   title={The integers and topology},
   book={
         title={Handbook of set-theoretic topology},
         editor={Kunen, Kenneth},
         editor={Vaughan, Jerry E.},
         publisher={North-Holland, Amsterdam},
         date={1984},
   },
   pages={111--167},
   review={\MR{776622}},
}

		
\bib{dhvmv2022}{article}{
author={Dow, Alan},
author={Hart, Klaas Pieter},
author={van Mill, Jan},
author={Vermeer, Hans},
title={Some realcompact spaces},
journal={Topology Proceedings}, 
volume={62},
date={2023}, 
pages={205--216},
note={E-published on August 25, 2023},
review={\MR{4634391}},
}

\bib{MR1039321}{book}{
   author={Engelking, Ryszard},
   title={General topology},
   series={Sigma Series in Pure Mathematics},
   volume={6},
   edition={2},
   note={Translated from the Polish by the author},
   publisher={Heldermann Verlag, Berlin},
   date={1989},
   pages={viii+529},
   isbn={3-88538-006-4},
   review={\MR{1039321}},
}

\bib{MR3003330}{article}{
   author={Fox, Keith M.},
   title={Generalizing witnesses to the non-normality of
   $\mathbb{N}^{\omega_1}$},
   journal={Topology Appl.},
   volume={160},
   date={2013},
   number={2},
   pages={337--340},
   issn={0166-8641},
   review={\MR{3003330}},
   doi={10.1016/j.topol.2012.11.011},
}


\bib{MR2516236}{article}{
   author={Hart, Joan E.},
   author={Kunen, Kenneth},
   title={Aronszajn compacta},
   journal={Topology Proc.},
   volume={35},
   date={2010},
   pages={107--125},
   issn={0146-4124},
   review={\MR{2516236}},
}

\bib{hvm2022}{article}{
author={Hart, Klaas Pieter},
author={van Mill, Jan},
title={Problems on $\beta\N$},
date={22 December 2022},
doi={10.48550/arXiv.2205.11204}
}

\bib{MR4138431}{article}{
   author={Hirata, Yasushi},
   author={Yajima, Yukinobu},
   title={$C^*$-, $C$- and $P$-embedded subsets in products and the
   undecidability of a certain property on $\mathbb{N}^{\omega_1}$},
   journal={Topology Appl.},
   volume={283},
   date={2020},
   pages={107350, 16},
   issn={0166-8641},
   review={\MR{4138431}},
   doi={10.1016/j.topol.2020.107350},
}


\bib{zbMATH03161747}{article}{
 author = {Kat{\v{e}}tov, Miroslav},
 title = {Remarks on characters and pseudocharacters},
 journal = {Commentat. Math. Univ. Carol.},
 ISSN = {0010-2628},
 Volume = {1},
 Number = {1},
 Pages = {20--25},
 Year = {1960},
 Language = {English},
}	

\bib{MR597342}{book}{
   author={Kunen, Kenneth},
   title={Set theory. An introduction to independence proofs},
   series={Studies in Logic and the Foundations of Mathematics},
   volume={102},
   publisher={North-Holland Publishing Co., Amsterdam-New York},
   date={1980},
   pages={xvi+313},
   isbn={0-444-85401-0},
   review={\MR{597342}},
}


\bib{MR3239211}{article}{
   author={Pol, El\.{z}bieta},
   author={Pol, Roman},
   title={Note on countable closed discrete sets in products of natural
   numbers},
   journal={Topology Appl.},
   volume={175},
   date={2014},
   pages={65--71},
   issn={0166-8641},
   review={\MR{3239211}},
   doi={10.1016/j.topol.2014.07.005},
}

\bib{MR776625}{article}{
   author={Todor\v{c}evi\'{c}, Stevo},
   title={Trees and linearly ordered sets},
   book={
      title={Handbook of set-theoretic topology},
      editor={Kunen, Kenneth},
      editor={Vaughan, Jerry E.},
      publisher={North-Holland, Amsterdam},
      date={1984},
   },
   pages={235--293},
   review={\MR{776625}},
}
	
\bib{MR1355064}{article}{
   author={Todor\v{c}evi\'{c}, Stevo},
   title={The functor $\sigma^2X$},
   journal={Studia Math.},
   volume={116},
   date={1995},
   number={1},
   pages={49--57},
   issn={0039-3223},
   review={\MR{1355064}},
   doi={10.4064/sm-116-1-49-57},
}
	
  \end{biblist}
\end{bibdiv}
\end{document}